\documentclass[11pt,letterpaper]{amsart}
\usepackage{amscd}
\usepackage{amssymb}
\usepackage{a4wide}
\usepackage{amstext}
\usepackage{amsthm}
\usepackage{xcolor}
\usepackage{cite}
\usepackage[T1,T2A]{fontenc}
\usepackage[utf8]{inputenc}
\usepackage{url}
\usepackage{amsfonts}
\usepackage{amssymb, amsthm}
\usepackage{amsmath}
\usepackage{mathtools}
\usepackage{needspace}
\usepackage[pdftex]{graphicx}
\usepackage{hyperref}
\usepackage{datetime}
\usepackage{epigraph}
\usepackage{verbatim}
\usepackage{mathtools}
\usepackage{xcolor}
\numberwithin{equation}{section}

\newcommand{\Z}{\mathbb{Z}}

\newcommand{\N}{\mathbb{N}}
\newcommand{\R}{\mathbb{R}}

\newcommand{\eps}{\varepsilon}

\DeclareMathOperator{\supp}{supp} 

\DeclareMathOperator{\beqq}{\begin{equation}} 

\DeclareMathOperator{\eeqq}{\end{equation}}

\renewcommand{\phi}{\varphi}

\newcommand{\beq}{\begin{equation}}
\newcommand{\eeq}{\end{equation}}

\newtheorem{Thm}{Theorem}[section]
\newtheorem{theorem}[Thm]{Theorem}

\newtheorem{corollary}[Thm]{Corollary}

\newtheorem{definition}{Definition}

\begin{document}

\sloppy
\title[Gabor frames for functions supported on a semi-axis]
{Gabor frames for functions supported on a semi-axis}

\author{Yurii Belov}
\address{Yurii Belov,
\newline Department of Mathematics and Computer Science
St. Petersburg State University,
14th Line 29b, Vasilyevsky Island, St. Petersburg, Russia, 199178,
\newline {\tt j\_b\_juri\_belov@mail.ru} }
\author{Aleksei Kulikov}
\address{Aleksei Kulikov,
\newline University of Copenhagen, Department of Mathematical Sciences,
Universitetsparken 5, 2100 Copenhagen, Denmark,
\newline {\tt lyosha.kulikov@mail.ru} 
}

\thanks{}

\begin{abstract} 
Let $g\in L^2(\mathbb{R})$ be a strictly decreasing continuous function supported on $\mathbb{R}_+$ such that for all $t > 0$ we have $g(x+t)\le q(t)g(x)$ for some $q(t)<1$. We prove that the Gabor system  
$$\mathcal{G}(g;\alpha,\beta):=\{e^{2\pi i \beta m x}g(x-\alpha n)\}_{m,n\in\mathbb{Z}}$$
always forms a frame in $L^2(\mathbb{R})$ for all lattice parameters $\alpha$,$\beta$, $\alpha\beta\leq 1$.
\end{abstract}

\maketitle
\section{Introduction}
For a function $g\in L^2(\R)$ and $t, \omega\in\R$ we define its time-frequency shift $\pi_{t,\omega}g$ by
$$\pi_{t,\omega}g(x)=e^{2\pi i \omega x}g(x-t).$$
For numbers $\alpha,\beta>0$ we consider the Gabor system
 $$\mathcal{G}(g;\alpha,\beta):=\{\pi_{\alpha m, \beta n }g\}_{m,n\in\mathbb{Z}}=\left\{e^{2\pi i \beta m x}g(x-\alpha n)\right\}_{m,n\in\mathbb{Z}}$$
and ask when it is a frame in $L^2(\R)$, that is when there are constants $A, B > 0$ such that for all $h\in L^2(\R)$ we have
\begin{equation}\label{frameform}
A\|h\|^2\leq\sum_{m,n\in\mathbb{Z}}|(h,\pi_{\alpha m, \beta n}g)|^2\leq B\|h\|^2.    
\end{equation}

The problem of classifying the frame set, that is the set of pairs $\alpha, \beta > 0$ such that $\mathcal{G}(g;\alpha,\beta)$ is a frame, is well studied \cite{BC, Gro, Heil, Jans5, RonShen, GroTP}. Despite numerous efforts, the full description is known only for a few specific functions: Gaussian, one-sided and two-sided exponential functions, hyperbolic secant, Haar function, shifted sinc-fucntion, characteristic function of an interval and three classes of functions: totally positive functions of finite type, Gaussian totally positive functions of finite type and Herglotz rational functions, see \cite{BB, BKL, BS, DS, DZ, Jans2, Jans3, JansStr, Gro1, Gro2, L, S, SW}. 

In this note we find a new and very general class of functions for which we can completely describe the frame set. This set first of all can not be parametrized by finitely many variables, which was the case for the results in \cite{BKL, Gro1, Gro2}, and second of all it completely recovers the previous results  \cite[Theorem 1.1, Theorem 1.10]{BKL},  \cite[Theorem 1.3]{BS}, \cite[Theorem C, p. 170] {Jans2} with a simple and streamlined proof. To state it we would need the following definition. 
\smallskip
\begin{definition} The function $g:[0, \infty)\to [0, \infty)$ will be called stably decreasing if for any $t>0$ there exists $0<q(t)<1$ such that
\begin{equation}
g(x)q(t)\ge g(x+t), \quad x\geq 0.
\label{qineq}
\end{equation}
\end{definition}

It is known that if $\alpha\beta>1$, then Gabor system $\mathcal{G}(g;\alpha,\beta)$ is not a frame, see \cite[Corollary 7.5.1]{Gro}. So, we will focus on the case $\alpha\beta\leq1$. Our main result is the following theorem.

\begin{theorem} Let $g\in L^2(\mathbb{R})$, $g(x)=0$ for $x<0$ and $g$ is stably decreasing. Let $x_0$ be the supremum of $x \ge 0$ such that $g(x) > 0$.
\begin{enumerate}
\item If $x_0 = \infty$ then the Gabor system $\mathcal{G}(g; \alpha, \beta)$ is a frame if and only if $\alpha\beta \le 1$.
\item If $x_0 < \infty$ and $\lim_{x\to x_0-}g(x) > 0$ then the Gabor system $\mathcal{G}(g; \alpha, \beta)$ is a frame if and only if $\alpha\beta \le 1$ and $\alpha \le x_0$.
\item If $x_0 < \infty$ and $\lim_{x\to x_0-}g(x) = 0$ then the Gabor system $\mathcal{G}(g; \alpha, \beta)$ is a frame if and only if $\alpha\beta \le 1$ and $\alpha < x_0$.
\end{enumerate}
\label{mainth}
\end{theorem}
\medskip

If the function $g$ has compact support, that is if $x_0 < \infty$ then $g$ is stably decreasing if and only if it is strictly decreasing. So, in particular, we described the frame sets for all compactly supported strictly decreasing functions, e.g. $(1-x)\chi_{[0, 1]}(x)$ and $e^{-x}\chi_{[0, 1]}(x)$. Note that only assuming that $g$ is compactly supported and non-decreasing is not enough as the frame set for $\chi_{[0, 1]}(x)$ is famously very complicated \cite{DS}.

Theorem \ref{mainth} also gives a lot of new examples of functions $g$ with the maximal possible frame set. For example, we cover all Cauchy transforms of positive finite measures supported on $[a, \infty)$ for $a > 0$.
\begin{corollary} Let $\mu$ be a positive finite measure on $[a,\infty)$, $a>0$. Put 
$$g(x)=\int_{a}^\infty\frac{d\mu(t)}{x-it}.$$
Then the Gabor system $\mathcal{G}(g;\alpha,\beta)$ is a frame for all $\alpha\beta\leq1$.
\label{Fcor}
\end{corollary}
If $\mu =\sum_{k=1}^N \mu_k\delta_{t_k}$, $\mu_k>0$, $t_k>0$ we get Theorem 1.1 from \cite{BKL} and if more generally $\mu = \sum_{k=1}^\infty \mu_k \delta_{t_k}$ for $t_k \ge a > 0$, $\mu_k> 0$, $\sum_{k=1}^\infty \mu_k < \infty$ then we cover Theorem 1.10 from \cite{BKL} as well.

The proof of Theorem \ref{mainth} is based on a Ron--Shen criterion and a modification of diagonal dominance theorem.

Our method is also applicable to the frame problem in the semi-irregular case. For a discrete set $\Gamma \subset \R^2$ we define the generalized Gabor system by
$$\mathcal{G}(g;\Gamma) := \{ \pi_{t, \omega}g\}_{(t, \omega)\in\Gamma}.$$

We will consider the sets $\Gamma$ of the form $\Lambda\times \beta \Z$ for some locally finite set $\Lambda\subset\R$. If $\Lambda$ is a $\frac{\alpha}{2}$-dense set for some $\alpha\le \frac{1}{\beta}$ then we can establish the frame property for such a system if $g$ is stably decreasing. Specifically, we prove the following theorem.

\begin{theorem} Let $\alpha, \beta > 0$ with $\alpha\beta\le1$. Let $g\in L^2(\mathbb{R})$, $g(x)=0$ for $x<0$, $g$ is stably decreasing and $\lim_{x\to \alpha^{-}} g(x) > 0$. Let $\Lambda=\{\lambda_n\}_{n\in\mathbb{Z}}\subset\mathbb{R}$ be such that for each $n\in \Z$ we have $0<\lambda_{n+1}-\lambda_n\leq \alpha$ and such that for some $m\in \N$ and all $n\in \Z$ we have $\lambda_{n+m}-\lambda_n \ge 1$. Then  the system $\mathcal{G}(g; \Lambda\times\beta\mathbb{Z})$ is a frame.
\label{irrth}
\end{theorem}
For the case of the one-sided exponential $f(x) = e^{-x}\chi_{[0, \infty)}(x)$ we also know that these conditions are necessary for the system $\mathcal{G}(g;\Lambda\times \beta\mathbb{Z})$ to be a frame, see \cite{BKL2}. The proof of Theorem \ref{irrth} is based on a new Ron--Shen type criterion (Theorem \ref{newron}) that we prove.

\section{Proof of Theorem \ref{mainth}}
We begin by recalling the Ron--Shen criterion for the Gabor system to be a frame (see, e.g., \cite[Proposition 6.3.4]{Gro}).
\begin{theorem}
Let $g\in L^2(\R)$ and consider the infinite matrix $$G_x = \left\{g\left(x - \alpha n + \frac{m}{\beta}\right)\right\}_{m,n\in\mathbb{Z}}.$$ The system $\mathcal{G}(g;\alpha,\beta)$ is a frame if and only if there are $A, B > 0$ such that for almost all $x\in \R$ and all $v\in \ell^2(\Z)$ we have
\begin{equation}
A||v||_2 \le ||G_x v||_2 \le B||v||_2.
\label{meq}
\end{equation}
\label{RS}
\end{theorem}
\subsection*{Proof of Theorem \ref{mainth}} First of all we notice that $g$ is non-increasing because $g(x)\ge g(x)q(t)\ge g(x+t)$ for all $x, t\ge 0$. Thus, $g$ is exponentially decreasing at infinity and bounded by $g(0)$. Therefore, the upper estimate in \eqref{meq} is obvious and so we will only focus on the lower estimate $A\|v\|_2\leq\|G_x v\|$.\smallskip

Recall that $\supp g\subset[0,\infty)$. Hence, the $n$-th row of $G_x$ has the following form
$$\begin{pmatrix}.&.&0&0 & 0& g\left(\tau_x\right) & g\left(\tau_x+\frac{1}{\beta}\right) & g\left(\tau_x+\frac{2}{\beta}\right)&.&.
\end{pmatrix},$$
where $\tau_x=x - \alpha n+\frac{m_n}{\beta}$, $m_n$ is the smallest integer such that $\tau_x\geq0$. Note that clearly $0 \le \tau_x < \frac{1}{\beta}$.
Since $\alpha\leq \frac{1}{\beta}$, we have $0\leq m_{n+1}-m_n\leq 1$. It may happen that $m_{n+1}=m_n$. These indices $n$ correspond to so-called "double" $ $rows. We note that for each $m\in \Z$ there exists $n\in\Z$ such that $m_n = m$ and $x - \alpha n + \frac{m_n}{\beta} < \alpha$. Indeed, pick any $n$ with $m_n = m$, if $x-\alpha n + \frac{m_n}{\beta} \ge \alpha$ then $m_{n-1}$ is still equal to $m$. We clearly can not indefinitely decrease $n$, hence sooner or later we will get that $x-\alpha n + \frac{m_n}{\beta} < \alpha$. Note that for each $m$ there is exactly one such $n$.

Let us remove all other rows from the matrix $G_x$. This can make the proof of the lower estimate in \eqref{meq} only harder.  Thus, it is sufficient to prove the lower estimate for the upper-triangular matrix $G^{\star}_x$,

$$G^{\star}_x=
\begin{pmatrix}
0&0&0 &  g\left(\tau_{x,n}\right) & g\left(\tau_{x,n}+\frac{1}{\beta}\right) & g\left(\tau_{x,n}+\frac{2}{\beta}\right)&.&.&.\\
0&0 & 0& 0 & g\left(\tau_{x,n+1}\right) & g\left(\tau_{x,n+1}+\frac{1}{\beta}\right) & g\left(\tau_{x,n+1}+\frac{2}{\beta}\right) &.&.\\
0&0&0&0 & 0& g\left(\tau_{x,n+2}\right) & g\left(\tau_{x,n+2}+\frac{1}{\beta}\right) & g\left(\tau_{x,n+2}+\frac{2}{\beta}\right) & .
\end{pmatrix},
$$
where the numbers $\tau_{x, n}$ are the ones corresponding to the remaining rows. Note that they are not equal to $x-\alpha n + \frac{m_n}{\beta}$ but to $x-\alpha k_n + \frac{m_{k_n}}{\beta}$ for some sequence $k_n\in \Z$.

\medskip

Let $q=q\left(\frac{1}{\beta}\right)$ be the number from \eqref{qineq}. Put $U={\rm{Id}}-qS$, where $S$ is the shift operator,
$$U={\rm{Id}}-qS,\quad U=\{b_{n,m}\},\quad b_{n,n}=1, b_{n,n+1}=-q, \quad b_{n,m}=0, m-n\notin\{0,1\}.$$
It is well-known that for $0<q<1$ the operator $U$ is invertible, $U^{-1}=\sum_{n=0}^\infty q^n S^n$. Hence, it is sufficient to prove that the matrix $G^{\star}_x U$ is invertible with a uniform (with respect to $x$) estimate for the norm of the inverse matrix $(G^{\star}_x U)^{-1}$. Let us consider the entries of the matrix $G^{\star}_x U=\{d_{n,m}\}$. We have 
$$d_{n,m}=0 \text { for } n < m,$$
$$d_{n,n} = g(\tau_{x,n}),$$
$$d_{n,m}= g\biggl{(}\tau_{x,n}+\frac{m-n}{\beta}\biggr{)}-q g\biggl{(}\tau_{x,n}+\frac{m-n-1}{\beta}\biggr{)}, \text{ for } m>n.$$
We have $d_{n,n}\ge 0$ and $d_{n,m}\le 0$, $m>n$, since function $g$ is stably decreasing. Therefore,  
$$\sum_{m=n+1}^\infty |d_{n,m}|=-\sum_{m=n+1}^\infty d_{n,m}=q g(\tau_{x,n})+\sum_{l=1}^\infty(q-1)g\biggl{(}\tau_{x,n}+\frac{l}{\beta}\biggr{)}$$
$$\leq q g(\tau_{x,n})=q d_{n,n}.$$

To finish the proof, we will use the following theorem which we prove in Section \ref{Dom}.
\begin{theorem} Let $D=\{d_{n,m}\}_{{n,m}\in\mathbb{Z}}$. Assume that for some $\delta > 0$, $0<\lambda<1$ and $C>0$ we have $|d_{n, n}|\ge \delta$, $|d_{n,m}|\leq C \lambda^{|m-n|}$ and 
$$\sum_{m\neq n}|d_{n,m}| \le \lambda |d_{n,n}|.$$
Then for some $\varepsilon = \eps(\delta, C, \lambda)>0$ we have $\|Dv\|\geq\varepsilon\|v\|$ for all $v\in\ell^2(\Z)$.
\label{domth}
\end{theorem}
We will apply this theorem with $\lambda = q(\frac{1}{\beta})$ and $C = \frac{g(0)}{q(\frac{1}{\beta})}$. The domination condition we already verified above. We have $g(x) \le g\left(\frac{[x\beta]}{\beta}\right)\le g(0) q(\frac{1}{\beta})^{[x\beta]}$, so the uniform exponential decay is also verified. Finally, we need a uniform lower bound $|d_{n, n}|\ge \delta > 0$. For this we recall that $d_{n, n} = g(\tau_{x, n})$ and that $0 \le \tau_{x, n} < \alpha$. Thus, if the support of $g$ is $\R_+$ then it is always at least $g(\alpha) > 0$, if the support of $g$ is $[0, x_0]$ and $\alpha < x_0$ then it is also at least $g(\alpha) > 0$ and if the support of $g$ is $[0, x_0]$ and $\alpha = x_0$ then it is at least $\lim_{x\to x_0-} g(x)$. So, we showed that all pairs $(\alpha, \beta)$ stated in Theorem \ref{mainth} give us a frame. It remains to show that all other pairs do not give us a frame.

If $\alpha\beta > 1$ then the system is never a frame. If $\alpha > x_0$ then the supports of the functions in $\mathcal{G}(g;\alpha, \beta)$ do not cover the whole $\R$ so this system is not even complete in $L^2(\R)$. Finally, if $\alpha = x_0$ and $\lim_{x\to x_0-} g(x)=0$ then one can check that the lower bound in \eqref{frameform} is violated for the functions $\chi_{[x_0-\eps, x_0]}$ as $\eps \to 0$. 

\qed

\smallskip

It seems that Theorem \ref{domth} is known in some form. If in addition we also have column dominance then this result can be found in \cite[Lemma 6.5.4]{Gro}. For the reader's convenience, we will give a proof of Theorem \ref{domth} in Section \ref{Dom}. It is interesting to note that without some uniform bound like $|d_{n,m}|\leq C \lambda^{|m-n|}$ Theorem \ref{domth} is no longer true.

\subsection{Proof of Corollary \ref{Fcor}} Let us consider the Fourier transform of $g$, for $\xi\geq0$ we have
$$\hat{g}(\xi)=\int_{\mathbb{R}}g(t)e^{-2\pi i \xi t} dt= \int_{a}^\infty e^{-2\pi\xi t }d\mu(t).$$
On the other hand, $\hat{g}(\xi)=0$ for $\xi<0$. So, the function $\hat{g}$ satisfies all of the assumptions of Theorem \ref{mainth} with $q(t) = e^{-2\pi t a}$. In addition, the Gabor system $\mathcal{G}(g;\alpha,\beta)$ forms a frame if and only if the Gabor system $\mathcal{G}(\hat{g};\beta,\alpha)$ forms a frame. Hence, we get the result. \qed
\section{Proof of Theorem \ref{irrth}}
\subsection{Ron--Shen type matrix} We begin by establishing the following theorem which gives us a criterion for the system $\mathcal{G}(g;\Gamma)$ to be a frame if $\Gamma$ is $(0,\beta)$-invariant. Note that $\Gamma = \Lambda\times \beta\Z$ satisfies this assumption.
\begin{theorem}\label{newron}
Let $\beta > 0$, $g\in L^2(\R)$, $\Lambda\subset \R$ and $c_\lambda \in \R, \lambda\in\Lambda$ be a sequence of numbers. Put 
$$\Gamma = \cup_{\lambda\in\Lambda}\{\lambda\}\times (\beta\Z +c_\lambda)\subset\R^2.$$
For $x\in\R$ consider the infinite matrix
$$G_x = \left\{g\left(x+ \frac{n}{\beta} - \lambda\right)e^{2\pi i c_\lambda \frac{n}{\beta}}\right\}_{\lambda\in\Lambda, n\in\Z}.$$
The system $\mathcal{G}(g;\Gamma)$ is a frame for $L^2(\R)$ if and only if there exist $A, B > 0$ such that for almost all $x\in \R$ and all $v\in \ell^2(\Z)$ we have
$$\frac{A}{\beta}\|v\|_{\ell^2(\Z)}^2\le \|G_x v\|_{\ell^2(\Lambda)}^2\le \frac{B}{\beta}\|v\|_{\ell^2(\Z)}^2.$$
Moreover, the best possible $A$ and $B$ are respectively the best lower and upper frame bounds for the system $\mathcal{G}(g;\Gamma)$.
\end{theorem}
\begin{proof}
We begin by considering the frame operator for $f\in L^2(\R)$
$$Sf = \sum_{n, \lambda} \langle f, \pi_{(\lambda, \beta n + c_\lambda)}g\rangle \pi_{(\lambda, \beta n + c_\lambda)}g.$$
The system $\mathcal{G}(g;\Gamma)$ is a frame if and only if $A{\rm Id}\le S \le B{\rm Id}$ for some $A, B>0$. We begin by writing $S$ explicitly
$$Sf(x) = \sum_{\lambda\in\Lambda}g(x-\lambda)e^{2\pi i c_\lambda x} \sum_{n\in\Z}e^{2\pi i \beta n x}\int_\R f(t)e^{-2\pi i (\beta n + c_\lambda)t}\bar{g}(t-\lambda)dt.$$
We notice that the integral is the value of the Fourier transform of the function $$h_\lambda(t)=f(t)e^{-2\pi i c_\lambda t}\bar{g}(t-\lambda)$$ at the point $\beta n$. Since the factor before it is an exponential $e^{2\pi i\beta n x}$, we can apply the Poisson summation formula to the inner sum and get
\begin{align*}Sf(x) = \frac{1}{\beta}\sum_{\lambda\in\Lambda}g(x-\lambda)e^{2\pi i c_\lambda x} \sum_{n\in\Z}h_\lambda\left(x+\frac{n}{\beta}\right) =\\ \frac{1}{\beta}\sum_{n\in\Z}f\left(x+\frac{n}{\beta}\right)\sum_{\lambda\in\Lambda}g(x-\lambda)\bar{g}\left(x+\frac{n}{\beta}-\lambda\right)e^{-2\pi i c_\lambda \frac{n}{\beta}}.
\end{align*}
We notice that for each $x\in \R$ the value $f(x)$ appears only with the values $f(y)$ for $y=x+\frac{s}{\beta},s\in\Z$. So, we have the frame property if and only if for almost all $x$ the matrix 
$$M_x (s, l) = \frac{1}{\beta}\sum_{\lambda\in\Lambda} g\left(x+\frac{s}{\beta}-\lambda\right)\bar{g}\left(x+\frac{l}{\beta}-\lambda\right)e^{-2\pi i c_\lambda \frac{l-s}{\beta}}$$
satisfies $A{\rm Id}\le M_x \le B{\rm Id}$, where we did the change of variables $n = l-s$. It remains to notice that
$$M_x = \frac{1}{\beta}G_x^*G_x,$$
so $M_x$ satisfies the required bounds if and only if $G_x$ satisfies our assumption and we also get the claim about the upper and lower frame bounds.
\end{proof}

\subsection{Proof of Theorem \ref{irrth}} We will apply Theorem \ref{newron} to the set $\Gamma = \Lambda\times \beta\Z$ so that $c_\lambda = 0$ for all $\lambda\in\Lambda$. We get the matrix
$$G_x = \left\{g\left(x-\lambda + \frac{k}{\beta}\right)\right\}_{\lambda\in \Lambda, k\in \Z}$$
which we think of as an operator from $\ell^2(\Z)$ to $\ell^2(\Lambda)$ and we want to show that it satisfies $$\frac{A}{\beta}\|v\|_{\ell^2(\Z)}^2\le \|G_x v\|_{\ell^2(\Lambda)}^2\le \frac{B}{\beta}\|v\|_{\ell^2(\Z)}^2,\quad A, B>0.$$
 Again, we will only focus on the lower bound as the upper bound follows from the rapid decay of the function $g$ and the separation of the sequence $\Lambda$. Just like in the proof of Theorem \ref{mainth}, we multiply $G_x$ by the matrix $U = {\rm{Id}}-q(\frac{1}{\beta})S$. In this way we will again get a matrix $D_x$ with the entries
$$d_{\lambda, n} = g\left(x - \lambda +\frac{n+1}{\beta}\right) - q\left(\frac{1}{\beta}\right)g\left(x-\lambda + \frac{n}{\beta}\right).$$ 
In the matrix $D_x$ the first element in each row is positive while each other element is negative, and the sum of negative elements in each row is bounded from below by $-q\left(\frac{1}{\beta}\right)$ times the positive element in this row.

To apply Theorem \ref{domth} we have to show that if we throw away some rows from the matrix $D_x$ it will become upper-triangular. That is, we have to show that for each column $n$ in the matrix $G_x$ there is at least one row for which the first element appears exactly in this column and it is uniformly bounded from below by $\lim_{x\to \alpha^{-}}g(x)$, meaning that there exists $\lambda\in \Lambda$ such that $\alpha > x-\lambda + \frac{n}{\beta} \ge 0 > x-\lambda + \frac{n-1}{\beta}$. This is  true because the differences between the consecutive elements of $\Lambda$ are at most $\alpha \le \frac{1}{\beta}$.\qed
\section{Dominance theorem\label{Dom}}
Let $D$ be a matrix satisfying the assumptions of Theorem \ref{domth}. Dividing each corresponding row by $d_{n, n}$ we can without loss of generality assume that $d_{n,n}=1$, $n\in\mathbb{Z}$ (note that this might potentially increase $C$). Let $\{x_n\}_{n}$ be an arbitrary vector from $\ell^2(\mathbb{Z})$. For any number $\kappa<1$ put
$$J=J(\kappa)=\{n\in\mathbb{Z}: |x_n|\geq \kappa^{|m-n|}|x_m| \text{ for any } m\}.$$
We have 
$$\biggl{(}1+\sum_{l\neq0}\kappa^{2|l|}\biggr{)} \times \sum_{n\in J}|x_n|^2\geq \sum_{n\in\mathbb{Z}}|x_n|^2.$$
This is so because for every $n\in\N$ there exists $m\in J$ such that $|x_n|\le \kappa^{|n-m|}|x_m|$. If $n\in J$ then this is obviously true with $m=n$, otherwise we can find some $m_1$ such that $|x_{n}|\le\kappa^{|n-m_1|}|x_{m_1}|$. If $m_1\in J$ then $m = m_1$ works. Otherwise, we will continue in this way constructing a sequence $m_1, m_2, \ldots $ of numbers all of which satisfy $|x_{m_k}|\le\kappa^{|m_k-m_{k+1}|}|x_{m_{k+1}}|$ and $m_k\neq m_{k+1}$. This would in particular imply that $|x_{m_k}|\kappa^k\ge |x_n|$ so the sequence $x$ is not bounded, which contradicts it being in $\ell^2(\Z)$. Therefore, at some point we have to stop at some $m_k\in J$ and $m = m_k$ works.

So, it is sufficient to prove that for some $\kappa<1$ and $\eps > 0$
$$\sum_{n\in J}|x_n+\sum_{m\neq n}d_{n,m}x_m|^2\ge\eps\sum_{n\in J}|x_n|^2.$$
For any $n\in J$ and $n_0\in\mathbb{N}$ we have by the triangle inequality
$$|x_n+\sum_{m\neq n}d_{n,m}x_m|$$ 
$$\geq |x_n|-\sum_{0<|m-n|\le n_0}|d_{n,m}|\kappa^{-n_0}|x_n|-C\sum_{|m-n|> n_0}\frac{\lambda^{|m-n|}}{\kappa^{|m-n|}}|x_n|.$$

We want to pick $\kappa < 1$, $n_0 \in \N$ such that 
$$q(n_0,\kappa, C):=\lambda \kappa^{-n_0}+2C\sum_{l=n_0+1}^\infty\biggl{(}\frac{\lambda}{\kappa}\biggr{)}^l < 1.$$
When $\kappa\to 1$, $q(n_0, \kappa, C)$ tends to $\lambda + 2C\sum_{l=n_0+1}^\infty \lambda^l$. This expression in turn tends to $\lambda < 1$ when $n_0\to \infty$. So, we first pick $n_0$ so that this expression is less than $1$ and then pick $\kappa$ close enough to $1$ so that $q(n_0, \kappa, C)$ is still less than $1$. We get
 \begin{multline*}\|D x_n\|_2^2 = \sum_{n\in \Z}|x_n+\sum_{m\neq n}d_{n,m}x_m|^2\\ \ge \sum_{n\in J}|x_n+\sum_{m\neq n}d_{n,m}x_m|^2\ge (1-q(n_0, \kappa, C))^2\sum_{n\in J}|x_n|^2\\ \ge (1-q(n_0, \kappa, C))^2\biggl{(}1+\sum_{l\neq0}\kappa^{2|l|}\biggr{)}^{-1}\sum_{n\in \Z}|x_n|^2.
\end{multline*}
\qed
\section*{Acknowledgments}  Yurii Belov was supported by the RSF grant 24-11-00087. Aleksei Kulikov was supported by the VILLUM Centre of Excellence for the Mathematics of Quantum Theory (QMATH) with Grant No.10059. 
\smallskip

Yurii Belov is the winner of the “Leader” competition conducted by the Foundation for the Advancement of Theoretical Physics and Mathematics “BASIS” and would like to thank its sponsors and jury.

\end{document}